\documentclass[12pt]{article}
\usepackage[top=2.54cm, bottom=2.54cm, left=2.80cm, right=2.80cm]{geometry}
\usepackage[tbtags]{amsmath}
\usepackage{amssymb}
\usepackage{amsthm}
\usepackage{fancyhdr}
\usepackage{latexsym}
\usepackage{mathrsfs}
\usepackage{wasysym}
\usepackage{fancyhdr}
\usepackage{float}
\usepackage{graphicx}
\usepackage[numbers,sort&compress]{natbib}
\usepackage{setspace}
\allowdisplaybreaks[4]

\newtheorem{theorem}{\bf Theorem}[section]
\newtheorem{lemma}[theorem]{Lemma}

\newtheorem{Def}[theorem]{Definition}

\newtheorem{corollary}[theorem]{Corollary}

\begin{document}
\begin{spacing}{1.1}
\title{A homogeneous polynomial associated with general hypergraphs and its applications \thanks{This research is supported by the National
Natural Science Foundation of China (Grant No.~11471077).}}
\author{ Yuan Hou$^{a, b}$\thanks {E-mail addresses: houyuan@fzu.edu.cn(Y. Hou),
  anchang@fzu.edu.cn(A. Chang).}, \, An Chang$^b$, \, Lei Zhang$^b$\\
\small$^a$Department of Computer Engineering, Fuzhou University Zhicheng College,\\
\small Fuzhou, Fujian, P.R. China\\
\small$^b$Center for Discrete Mathematics and Theoretical Computer Science,\\
\small Fuzhou University, Fuzhou, Fujian, P.R. China}
\date{}
\maketitle
{\bf Abstract}\,: 
In this paper, we define a homogeneous polynomial for a general hypergraph, and establish a remarkable connection between clique number and the homogeneous polynomial of a general hypergraph. For a general hypergraph, we explore some inequality relations among spectral radius,
clique number and the homogeneous polynomial.
We also give lower and upper bounds on the spectral radius in terms of the clique number.

{\bf AMS}\,: 15A42; 05C50.


\section{Introduction}

\noindent

A general hypergraph is a pair $H=(V, E)$ consisting of a vertex set $V$ and
an edge set $E$, where each edge is a subset of $V$. If $H'=(V', E')$ is a hypergraph such that $V'\subseteq V$ and $E'\subseteq E$, then $H'$ is called a subhypergraph of $H$. For a vertex $i$, we denote the family of edges containing $i$ by $E(i)$. Two vertices $i$ and $j$ are said to be adjacent, denoted by $i\sim j$, if there
exists an edge $e$ such that $\{ i,j \}\subseteq e$. Otherwise, we call $i$ and $j$ are nonadjacent.

The set $R=\{|e|: e\in E \}$ is called the set of edge types of $H$. We also say that $H$ is an $R$-graph.
In particular, if $R$ contains only one positive integer $r$, then an $\{r\}$-graph is just an $r$-uniform hypergraph, which is simply written
as $r$-graph. Consequently, a $2$-graph is referred to a usual graph.  A hypergraph is non-uniform
if it has at least two edge types. For a vertex $i$, let $R(i)$ be the multiset of edge types in $E(i)$. For example, assume $E(1)=\{\{1,3\}, \{1,2,3\}, \{1,3,4\}\}$, then $R(1)=\{2,3,3\}$.
The rank of $H$, denoted by $rank(H)$, is the maximum cardinality of the edges in the hypergraph. For example, if $R=\{1,4\}$, we say that $H$ is a $\{1,4\}$-graph with $rank(H)=4$.

For an integer $n$, let $[n]$ denote the set $\{1,2, \cdots, n\}$. For a set
$S$ and integer $i$, let $\left(
  \begin{array}{ccc}
    S\\
    i
  \end{array}
\right)$
be the family of all $i$-subsets of $S$. An $R$-graph $H$ with vertex set $[n]$ and edge set $\bigcup\limits_{i\in R}
\left(
  \begin{array}{ccc}
    [n]\\
    i
  \end{array}
\right)$ is called a complete $R$-graph. A complete $R$-subgraph in $H$ is called a clique of $H$. A clique is said
to be maximal if it is not contained in any other clique, while it is called maximum if
it has maximum cardinality. The clique number of a hypergraph $H$, denoted by $\omega(H)$, is
defined as the number of vertices of a maximum clique. In other words, the clique number of a hypergraph $H$ is the number of vertices of its maximum complete
$R$-subgraph in $H$. In particular, when the edge set of $H$ is empty, we delimit $\omega(H)=1$ for the purpose of complying with mathematical logic.

In 1967, Wilf \cite{Wilf1967} first used spectral graph theory for computing
bounds on the chromatic number of graphs.
\begin{theorem}[\cite{Wilf1967}]\label{Wilf}
 Let $G$ be a 2-graph with chromatic number $\chi(G)$ and spectral
radius $\rho(G)$. Then
$$\chi(G)\leq \rho(G)+1. $$
\end{theorem}

By above Wilf's result, it is immediately to obtain that $\omega(G)\leq \rho(G)+1$. Later in 1986, Wilf \cite{Wilf1986} introduced a lower spectral bound on clique number which was inspired by an elegant result
due to Motzkin and Straus \cite{Motzkin1965}.

In 1965, Motzkin and Straus gave an answer to the
following problem proposed in \cite{Macdonald1963}.

{\it Given a graph $G = (V, E)$ with vertex set $V=\{1,2, \cdots, n\}$. Let
$S$ be the simplex in $\mathbb{R}^n$ given by $x_{i}\geq {0}, \sum\limits_{i=1}^{n} {x_{i}}=1$. What is $\max\limits_{x\in S}\sum\limits_{\{i,j\}\in{E}}x_{i}x_{j}$? }

Let $ L(G,x)=\sum\limits_{\{i,j\}\in{E}}x_{i}x_{j}$. The Motzkin and Straus Theorem establishes a link between the problem
of finding the clique number of a graph $G$ and the problem of optimizing a homogeneous polynomial $L(G,x)$ of $G$ over the simplex $S=\{ (x_{1},x_{2},\cdots,x_{n})|\sum\limits_{i=1}^{n} {x_{i}}=1, x_{i}\geq {0}\,\, for\,\, i=1,2,\cdots, n. \}$.

\begin{theorem}[Motzkin and Straus Theorem \cite{Motzkin1965}]\label{Motzkin-Straus}
 Let $G$ be a 2-graph with clique number $\omega(G)$,
and $x^{*}$ a maximizer of $L(G, x)$ over $S$. Then
$$L(G,x^{*})=\frac{1}{2}(1-\frac{1}{\omega(G)}). $$
\end{theorem}

\begin{theorem}[\cite{Wilf1986}]\label{Wilf1986}
 Let $G$ be a 2-graph with spectral
radius $\rho(G)$ and principal eigenvector $x$. Then
$${\omega(G)}\geq \frac{M^{2}}{M^{2}-\rho(G)}. $$
where $M$ is the sum of the entries of the principal eigenvector $x$.
\end{theorem}

Naturally we want similar comfort and convenience
for spectra of hypergraphs. So it
is a natural thought to generalize the Wilf's results to general hypergraphs by using the tool of the spectral hypergraph theory.

We first define a homogeneous polynomial of degree $m$ for general hypergraphs.

For a general hypergraph $H$ with $rank(H)=m$, and an edge $e=\{{l_1},{l_2}, \cdots, {l_s}\}$ with cardinality $s\leq m$, we define \begin{equation*}\label{edge polynomial}x^{e}_{m}=\sum x_{i_{1}}x_{i_{2}} \cdots x_{i_{m}},\end{equation*}
where the sum is over $i_{1}, i_{2}, \cdots, i_{m}$ chosen in all possible ways from $\{l_{1}, l_{2}, \cdots, l_{s}\}$ with at least once for each element of the set.

\begin{Def}\label{lagrange x}
Let $S =\{ (x_{1},x_{2},\cdots,x_{n})|\sum\limits_{i=1}^{n} {x_{i}}=1, x_{i}\geq {0}\,\, for\,\, i=1,2,\cdots, n. \}$ and $H$ be a  general hypergraph with $rank(H)=m$. For a vector $x$ in $S$, define \begin{equation}
L(H, x)=\sum\limits_{e\in{E}}\frac{1}{\alpha(s)}x_{m} ^{e}, \end{equation}
where $s$ is the cardinality of the edge $e$ and $\alpha(s)=\sum\limits_{k_1,\cdots,k_s\geq 1,\atop k_1+\cdots+k_s=m} \frac{m!}{k_1!k_2!\cdots k_s!}$.\\
Furthermore,
\begin{equation*}\label{L(H)} L(H)=\max\{L(H, x): x\in S \}.\end{equation*}
A vector $x\in S$ is called an optimal weighting for $H$ if $L(H)=L(H, x)$.
\end{Def}

Observe that if $H$ is a 2-graph, by Equation (1), we have $\alpha(2)=2$. For an arbitrary vector $x$ in $S$,
$$L(H, x)=\sum\limits_{\{i,j\}\in{E}}\frac{1}{2}(x_{i}x_{j}+x_{j}x_{i})=\sum\limits_{\{i,j\}\in{E}}x_{i}x_{j},$$
which is exactly the homogeneous polynomial for 2-graphs in the Motzkin-Straus theorem.

In this paper,
we apply the homogeneous polynomial methods to study some relations between the
largest $H$-eigenvalues of adjacency tensor and clique numbers of general hypergraphs.
This work is motivated by the classic results for graphs \cite{Wilf1967,Wilf1986}
and some recent results \cite{Frankl1998,Keevash2013, Mubayi2006, Rota 2009,Pelillo 2009,Talbot2002,Yuejian Peng2016,peng2016,chang2013}.
Notice that the Graph-Lagrangian of a non-uniform hypergraph is a nonhomogeneous polynomial, which is different from our definition here. In this paper, we also adopt some definitions and methods of previous studies.

The rest of this paper is organized as follows. In the next section, we present
some definitions and properties on eigenvalues of tensors and hypergraphs. Also we
give some useful tools to complete our proof. In Section 3, we attempt to explore the relationships among the homogeneous polynomial, the spectral radius, and the clique number for general hypergraphs. We also bound the spectral radius and clique number for general hypergraphs based on the results obtained.

\section{Preliminary}

\noindent

In 2005, Qi \cite{Qi2005_2} and Lim \cite{Lim2005_3} independently introduced the concept of tensor eigenvalues and the spectra of tensors.
An $m$th-order $n$-dimensional real tensor $\mathcal{T}=(\mathcal{T}_{i_{1}\cdots  i_{m}})$ consists of $n^{m}$ real entries $\mathcal{T}_{i_{1}\cdots  i_{m}}$ for $1\leq{i_{1}, i_{2},\cdots, i_{m}}\leq {n}$.
Obviously, a vector of dimension $n$ is a tensor of order $1$ and a matrix is a tensor of order $2$. $\mathcal{T}$ is called symmetric if the value of $\mathcal{T}_{i_{1}\cdots  i_{m}}$ is invariant under any permutation of its indices $i_{1}, i_{2},\cdots, i_{m}$. Given a vector $x \in R^{n}$, $\mathcal{T}x^{m}$ is a real number and $\mathcal{T} x^{m-1}$ is an $n$-dimensional vector. $\mathcal{T}x^{m}$ and the $i$th component of $\mathcal{T} x^{m-1}$ are defined as follows:
\begin{eqnarray*}
\mathcal{T}x^{m}&=&\sum_{i_{1},i_{2},\cdots, i_{m} \in [n]} \mathcal{T}_{i_{1}i_{2} \cdots i_{m}}x_{i_{1}}x_{i_{2}} \cdots x_{i_{m}}.\\
(\mathcal{T}x^{m-1})_{i}&=&\sum_{i_{2},\cdots, i_{m}\in [n]} \mathcal{T}_{ii_{2} \cdots i_{m}}x_{i_{2}} \cdots x_{i_{m}}.\end{eqnarray*}

Let $\mathcal{T}$ be an $m$th-order $n$-dimensional real tensor. For some $\lambda \in {\mathbb{C}}$, if there exists a nonzero vector $x \in \mathbb{C}^{n}$ satisfying the following eigenequation
\begin{equation*} \label{eigenequations} \mathcal{T}x^{m-1} = \lambda x^{[m-1]}. \end{equation*}
Then $\lambda$ is an eigenvalue of $\mathcal{T}$ and $x$ is its corresponding eigenvector, where $x^{[m-1]}:=({x^{m-1}_{1}},{x^{m-1}_{2}},$ $\cdots,{x^{m-1}_{n}})^{T} \in \mathbb{C}^{n}\setminus\{0\}$.

If $x$ is a real eigenvector of $\mathcal{T}$,
surely the corresponding eigenvalue $\lambda$ is real. In this case, $\lambda$ is called an $H$-eigenvalue and $x$ is called an
$H$-eigenvector associated with $\lambda$. Furthermore, if $x$ is nonnegative and real, we say $\lambda$ is an $H^{+}$-eigenvalue of $\mathcal{T}$. If $x$ is positive and real, $\lambda$ is said to be an $H^{++}$-eigenvalue of $\mathcal{T}$.
The maximal absolute value of the eigenvalues of $\mathcal{T}$ is called the spectral radius of $\mathcal{T}$, denoted by $\rho(\mathcal{T})$.

For nonnegative tensors, we have the Perron-Frobenius theorem, established as
\begin{theorem}[Perron-Frobenius theorem for nonnegative tensors \cite{Fan2015}]\label{Perron-Frobenius theorem for nonnegative tensors}\mbox{}\par
\noindent(1) (Yang and Yang 2010). If $\mathcal{T}$ is a nonnegative tensor of order $k$ and dimension $n$,
then $\rho(\mathcal{T})$ is an $H^{+}$-eigenvalue of $\mathcal{T}$.\\
(2) (Friedland Gaubert and Han 2011). If furthermore $\mathcal{T}$ is weakly irreducible, then
$\rho(\mathcal{T})$ is the unique $H^{++}$-eigenvalue of $\mathcal{T}$, with the unique eigenvector $x\in R_{++}^{n}$, up
to a positive scaling coefficient.\\
(3) (Chang Pearson and Zhang 2008). If moreover $\mathcal{T}$ is irreducible, then $\rho(\mathcal{T})$ is the
unique $H^{+}$-eigenvalue of $\mathcal{T}$, with the unique eigenvector $x\in R_{+}^{n}$, up to a positive
scaling coefficient.
\end{theorem}

In 2012, Cooper and Dutle \cite{Cooper2012} defined the adjacency tensor of an $r$-graph. Later in 2017, Banerjee et al.\cite{Banerjee2017} defined the adjacency tensor for general hypergraphs as the following.

\begin{Def}\label{generalhypergraphadjacency}Let $H=(V,E)$ be a general hypergraph with $rank(H)=m$. The adjacency tensor $\mathcal{A}$ of $H$ is defined as follows
\begin{equation*}\mathcal{A}=(a_{i_1i_2\cdots i_m}),1\leq i_1,i_2,\cdots,i_m\leq n.\end{equation*}
For all edges $e=\{{l_1},{l_2}, \cdots, {l_s}\}\in E$ of cardinality $s\leq m$,
\begin{equation*}\label{general adjacency entry}a_{i_1i_2\cdots i_m}=\frac{s}{\alpha(s)}, where \ \alpha(s)=\sum\limits_{k_1,\cdots,k_s\geq 1,\atop k_1+\cdots+k_s=m} \frac{m!}{k_1!k_2!\cdots k_s!}\end{equation*} and $i_1,i_2,\cdots,i_m$ are chosen in all possible ways from $\{l_1,l_2,\cdots,l_s\}$ with at least once for each element of the set. The other positions of the tensor are zeros.
\end{Def}

Notice that if $H$ is an $r$-graph, it is known from direct calculations that $a_{i_1i_2\cdots i_r}= \frac{1}{(r-1)!}$ for an arbitrary edge $e=\{{i_1},{i_2}, \cdots, {i_r}\}\in E$, which is exactly the definition of Cooper and Dutle in \cite{Cooper2012}.

For an edge $e=\{i,i_{2},\cdots,i_{s}\}$ in a general hypergraph with $rank(H)=m$, we denote an $m$ order $n$ dimensional symmetric tensor $\mathcal{A}(e)$ by
$$(\mathcal{A}(e)x)_i=\frac{s}{\alpha(s)} \sum\limits_{k_1\geq 0,k_2,\cdots,k_s\geq 1,\atop k_1+\cdots+k_s=m-1} \frac{(m-1)!}{k_1!k_2!\cdots k_s!} x_i^{k_1} x_{i_2}^{k_2} \cdots x_{i_s}^{k_s},$$
which indicates that
$$x^\top \mathcal{A}(e)x=\frac{s}{\alpha(s)} \sum\limits_{k_1,\cdots,k_s\geq 1,\atop k_1+\cdots+k_s=m} \frac{m!}{k_1!k_2!\cdots k_s!} x_{i}^{k_1} x_{i_2}^{k_2} \cdots x_{i_s}^{k_s}.$$
Then the adjacency tensor $\mathcal{A}$ of order $m$ and dimension $n$ uniquely defines a homogeneous polynomial in $n$ variables of degree $m$ by: $$F_{\mathcal{A}}(x)=\mathcal{A}x^{m}=\sum\limits_{e\in E} x^\top \mathcal{A}(e)x.$$

Based on the definnitions in \cite{Banerjee2017}, Kang et al.\cite{Kang2017} obtained the Perron-Frobenius theorem for general hypergraphs.
For convenience, let $\rho(H)$ denote the spectral radius of the adjacency tensor of a general hypergraph $H$.
\begin{theorem}[Perron-Frobenius theorem for general hypergraphs \cite{Kang2017}]\label{Perron-Frobenius theorem for general hypergraphs}\mbox{}\par
\noindent(1) Let $H$ be a general hypergraph, then $\rho(H)$ is an $H^{+}$-eigenvalue of $H$.\\
(2) If $H$ is connected, then $\rho(H)$ is the unique $H^{++}$-eigenvalue of $H$, with the unique eigenvector $x\in \mathbb{R}_{++}^{n}$, up to a positive scaling cofficient.
\end{theorem}

By Theorem~\ref{Perron-Frobenius theorem for general hypergraphs}, if $H$ is connected, there is a unique unit positive eigenvector $x$ corresponding to $\rho{(H)}$. Let $\| x \|_{m}=(\sum \limits_{i=1}^{n} x_{i}^{m})^{\frac{1}{m}}$. The positive eigenvector $x$ with $\| x \|_{m}=1$ corresponding to $\rho(H)$ is called the principal eigenvector of $H$. Assume $x$ is the principal eigenvector of an $R$-graph $H$,  by the theory of optimization, we have
\begin{equation}\label{spectral radius equation} \rho{(H)}=F_{\mathcal{A}}(x)=\mathcal{A}x^{m}=\sum\limits_{e\in E} x^\top \mathcal{A}(e)x.\end{equation}

We give some auxiliary lemmas which will be used in the sequel.
\begin{lemma}[Maclaurin's inequality \cite{Hardy.Littlewood1988_4}]\label{Maclaurin¡¯s inequality}
Let $x_{1}, x_{2}, \cdots,
x_{n}$ be positive real numbers. For any $k\in [n]$, define $S_{k}$ as follows:
\begin{equation*}\label{ Maclaurin's inequality}
S_{k}=\frac{\sum\limits_{1\leq i_{1}<i_{2}<\cdots<i_{k}\leq n}x_{i_{1}}x_{i_{2}} \cdots x_{i_{k}}}{ \left(
  \begin{array}{ccc}
    n\\
    k
  \end{array}
\right)}.\end{equation*}
Then $S_{1}\geq \sqrt{S_{2}} \geq \sqrt[3]{S_{3}}\cdots \geq \sqrt[n]{S_{n}}$ with equality if and only if all the $x_{i}$ are equal.\end{lemma}

For any $k\in [n]$, by Maclaurin's inequality, we get
$S_{1}\geq \sqrt[k]{S_{k}}$, that is,  $$\sum_{1\leq i_{1}<i_{2}<\cdots<i_{k}\leq n}x_{i_{1}}x_{i_{2}} \cdots x_{i_{k}}\leq { \left(
  \begin{array}{ccc}
    n\\
    k
  \end{array}
\right)}(\frac{x_{1}+x_{2}+\cdots+x_{n}}{n})^{k}.$$

The following lemma is the generalization of the Cauchy-Schwarz inequality for more than two vectors.
\iftrue
\begin{lemma}[\cite{Hardy.Littlewood1988_4}]\label{Cauchy¨CSchwarz inequality}
Let $x_{1}=(x_{i}^{(1)}), x_{2}=(x_{i}^{(2)}), \cdots,
x_{k}=(x_{i}^{(k)})$ be nonnegative vectors of dimension $n$.  Then
\begin{equation}\label{inequality}
\sum_{i=1}^{n}\prod_{j=1}^{k}x_{i}^{(j)} \leq \| x_{1} \|_{k}\| x_{2} \|_{k}\cdots| x_{k} \|_{k}.
\end{equation}
Equality holds if and only if all vectors are collinear to one of them.
\end{lemma}
\fi

\section{Main results}

\noindent

In this section, we discuss the relations between the clique number and the spectral radius of an $R$-graph. A tight lower bound of the spectral radius of an $R$-graph is presented. Further, we determine the upper bound of the spectral radius based on
 the clique number for an $\{m,m-1\}$-graph, which derived from a Motzkin-Straus type result due to $L(H)$ for $\{m,m-1\}$-graphs. In the following discussions, without loss of generality, suppose $m\geq3$.

\begin{theorem}\label{lower bound for R-hypergraph}
 Let $H$ be an $R$-graph with clique number $\omega$. Then
\begin{equation}\rho(H)\geq \sum\limits_{s\in R}\left(
  \begin{array}{ccc}
    \omega-1\\
    s-1
  \end{array}
\right) \end{equation}
with equality if and only if $H$ is a complete $R$-graph.

\end{theorem}

\begin{proof}
Assume $H_{0}$ be the maximum complete $R$-subgraph of $H$ on $\omega$ vertices. Suppose $rank(H)=m$. Let $x$ be a nonnegative vector of dimension $n$ with entries \[x_{i}=\begin{cases}
\frac{1}{\sqrt[m]{\omega}}&\text{if $i\in H_{0}$},\\
0&\text{otherwise}.
\end{cases}\] It is easy to verify that $\| x \|^{m} _{m}=1$.

For an arbitrary edge $e$ in $H_{0}$, suppose $|e|=s$, we have
\begin{eqnarray*}\label{clique eq1} x^\top \mathcal{A}(e)x&=&\frac{s}{\alpha(s)} \sum\limits_{k_1,\cdots,k_s\geq 1,\atop k_1+\cdots+k_s=m} \frac{m!}{k_1!k_2!\cdots k_s!} x_{i}^{k_1} x_{i_2}^{k_2} \cdots x_{i_s}^{k_s}\\&=&\frac{1}{\omega}\frac{s}{\alpha(s)}\sum\limits_{k_1,\cdots,k_s\geq 1,\atop k_1+\cdots+k_s=m} \frac{m!}{k_1!k_2!\cdots k_s!}\\&=&\frac{s}{\omega}.\end{eqnarray*}

After using the Equation (\ref{spectral radius equation}), we have
\begin{eqnarray*}
\rho(H)&\geq&F_{\mathcal{A}}(x)=\sum\limits_{e\in E(H_{0})} x^\top \mathcal{A}(e)x
=\sum\limits_{s\in R}\sum\limits_{e\in E(H_{0}),\atop |e|=s}\frac{s}{\omega}
= \sum\limits_{s\in R}\frac{s}{\omega}\left(
  \begin{array}{ccc}
    \omega\\
    s
  \end{array}
\right)
=\sum\limits_{s\in R}\left(
  \begin{array}{ccc}
    \omega-1\\
    s-1
  \end{array}
\right).
\end{eqnarray*}

If $\rho(H)= \sum\limits_{s\in R}\left(
  \begin{array}{ccc}
    \omega-1\\
    s-1
  \end{array}
\right)$, that is to say the unit vector $x$ is a maximizer of $F_{\mathcal{A}}(x)$ over $\mathbb{R}_{+}^{n}$. Then by Theorem \ref{Perron-Frobenius theorem for general hypergraphs},
$x$ must be a positive vector. Thus all vertices in $H$ belong
to $H_{0}$, which indicates that $\omega=n$, i.e., $H$ is a complete $R$-graph. On the other
hand, if $H$ is a complete $R$-graph, then $\omega=n$ and $x$ is the principal eigenvector. By Equation (\ref{spectral radius equation}), we have
$\rho(H)= \sum\limits_{s\in R}\left(
  \begin{array}{ccc}
    \omega-1\\
    s-1
  \end{array}
\right)$. Therefore the theorem follows.\end{proof}

Observe that if $H$ is an $m$-graph, Theorem \ref{lower bound for R-hypergraph} indicates that $\rho(H)\geq
\left(
  \begin{array}{ccc}
    \omega-1\\
    m-1
  \end{array}
\right) $
as proved by Yi and Chang in \cite{chang2013}.

Spectral methods for 2-graphs reside on a solid ground, with traditions
settled both in tools and problems, such as number of edges, independence number.
Theorem \ref{lower bound for R-hypergraph} is also a useful tool for general $R$-graphs.  Now, we give the lower bound of clique number for $R$-graphs with $R=\{ m, m-1 \}$.
\begin{theorem}
 Let $H$ be an $\{m,m-1\}$-graph with clique number $\omega$. Then
$$\omega \leq m-2+[(m-1)!\rho(H)]^{\frac{1}{m-1}}. $$
\end{theorem}
\begin{proof} Since
$$\left(
  \begin{array}{ccc}
    \omega\\
    m-1
  \end{array}
\right)=\frac{\omega(\omega-1)\cdots(\omega-m+2)}{(m-1)!}\geq \frac{(\omega-m+2)^{m-1}}{(m-1)!}
$$
By Theorem \ref{lower bound for R-hypergraph}, if $R=\{m,m-1\}$, we have
\begin{eqnarray*}
\rho(H)&\geq& \left(
  \begin{array}{ccc}
    \omega-1\\
    m-1
  \end{array}
\right)+ \left(
  \begin{array}{ccc}
    \omega-1\\
    m-2
  \end{array}
\right)= \left(
  \begin{array}{ccc}
    \omega\\
    m-1
  \end{array}
\right)\geq \frac{(\omega-m+2)^{m-1}}{(m-1)!}.
\end{eqnarray*}

Thus it is easy to verify that $\omega \leq m-2+[(m-1)!\rho(H)]^{\frac{1}{m-1}}.$ This completes the proof.
\end{proof}

\begin{theorem}\label{general M-S}
Let $H$ be an $\{m,m-1\}$-graph with clique number $\omega$. If either $H$ is a complete $\{m,m-1\}$-graph or there exists two nonadjacent vertices $i$ and $j$ such that $R(i)=R(j)$, then
\begin{eqnarray}L(H)&=&(\frac{1}{\omega})^{m}\left(
  \begin{array}{ccc}
    \omega+1\\
    m
  \end{array}
\right).\end{eqnarray}
\end{theorem}

\begin{proof} Assume $H_{0}$ be the maximum complete $\{m,m-1\}$-subgraph of $H$ on $\omega$ vertices. Let $x$ be a nonnegative vector of dimension $n$ with entries \[x_{i}=\begin{cases}
\frac{1}{\omega}&\text{if $i\in H_{0}$},\\
0&\text{otherwise}.
\end{cases}\]  It is easy to verify that $\| x \|_{1}=\sum \limits_{i=1}^{n} x_{i}=1$, i.e., $x\in S$.

After using the Definition \ref{lagrange x}, we have
\begin{eqnarray*}
L(H)&\geq& L(H,x)\\&=& \sum\limits_{e\in{E(H_{0})}}\frac{1}{\alpha(s)}x_{m} ^{e}\\
&=& \sum\limits_{s\in \{m,m-1\}}\frac{1}{\alpha(s)}\sum\limits_{ e\in E(H_{0}),\atop |e|=s}x^{e}_{m}\\
&=&\sum\limits_{s\in \{m,m-1\}}\frac{1}{\alpha(s)}\sum\limits_{e\in E(H_{0}),\atop |e|=s}\alpha(s)({\frac{1}{\omega}})^{m}\\
&=&({\frac{1}{\omega}})^{m}\sum\limits_{s\in \{m,m-1\}}\left(
  \begin{array}{ccc}
    \omega\\
    s
  \end{array}
\right)\\
&=&(\frac{1}{\omega})^{m}\left(
  \begin{array}{ccc}
    \omega+1\\
    m
  \end{array}
\right).\end{eqnarray*}

To prove the opposite inequality, we proceed by induction on $n$. For $n\leq m-2$,
we have $\omega =1$ and $L(H)=0$. Without loss of generality, suppose $n$ is sufficiently large. Assume the theorem true for $\{m,m-1\}$-graphs with
fewer than $n$ vertices.
Suppose $x$ is an optimal weighting for $H$, where $x_{1}\geq x_{2}\geq \cdots \geq x_{k} > x_{k+1}=x_{k+2}=\cdots=x_{n}=0$.

If $k<n$, there exists one of the entries $x_{i}=0$.
Let $H'$ be obtained from
$H$ by deleting the corresponding vertex $i$ and the edges containing $i$. Since the theorem holds for $H'$, we have
\begin{eqnarray*}
L(H)&=& L(H')= ({\frac{1}{\omega'}})^{m}\left(
  \begin{array}{ccc}
    \omega'+1\\
    m
  \end{array}
\right),\end{eqnarray*}
where $\omega'$ is the clique number of $H'$. It is clearly that $\omega' \leq \omega$.

For $\omega\geq 1$,
\begin{eqnarray*}
({\frac{1}{\omega}})^{m}\left(
  \begin{array}{ccc}
    \omega+1\\
    m
  \end{array}
\right)&=&\frac{(1+\frac{1}{\omega})(1-\frac{1}{\omega})(1-\frac{2}{\omega})\cdots(1-\frac{m-2}{\omega})}{m!}
=\frac{(1-\frac{1}{\omega^{2}})(1-\frac{2}{\omega})\cdots(1-\frac{m-2}{\omega})}{m!},
\end{eqnarray*}
which is an monotonically increasing function of $\omega$ . Thus \begin{eqnarray*}
L(H)&=& L(H')\leq ({\frac{1}{\omega}})^{m}\left(
  \begin{array}{ccc}
    \omega+1\\
    m
  \end{array}
\right).\end{eqnarray*}

If $k=n$, that is $x_{1}\geq x_{2}\geq \cdots \geq x_{n}>0$. To proceed our proof, we consider the following two cases.

{\bf{Case 1.}}
There exists two nonadjacent vertices $i$ and $j$ such that $R(i)=R(j)$. Notice that $R(i)$and $R(j)$ are multisets.
For a vector $x$ in $R_{+}^{n}$,
we write $$L^{i}(H,x)=\sum\limits_{e\in E(i)}\frac{1}{\alpha(s)}x_{m}^{e},$$
and
$$L^{j}(H,x)=\sum\limits_{e\in E(j)}\frac{1}{\alpha(s)}x_{m}^{e}.$$

Define two
vectors $y=(y_{k})$ and $z=(z_{k})$ in $R_{+}^{n}$ as follows. Let $y_{k} = x_{k}$ for $k\neq i,j$, $y_{i}=x_{i}+x_{j}$ and $y_{j}=0$. Clearly, $y\in S$.
Further, $z_{i} = x_{j}$ and $z_{k} = x_{k}$ for $k\neq i$.


Without loss of generality, assume $L^{i}(H,z)\geq L^{j}(H,x)$.
By definition \ref{lagrange x} and $R(i)=R(j)$,  we have
\begin{eqnarray*}
L(H,y)-L(H,x)&=& L^{i}(H,y)-L^{i}(H,x)-L^{j}(H,x)\\
&=& \sum\limits_{e\in E(i)}\frac{1}{\alpha(s)}y_{m}^{e}-\sum\limits_{e\in E(i)}\frac{1}{\alpha(s)}x_{m}^{e}
-\sum\limits_{e\in E(j)}\frac{1}{\alpha(s)}x_{m}^{e}.\\
\end{eqnarray*}

For an arbitrary positive interger $k$, it is clearly that $(x_{i}+x_{j})^{k}> x_{i}^{k}+x_{j}^{k}$.  Consequently, $L^{i}(H,y)>L^{i}(H,x)+L^{i}(H,z).$
Thus
$$
L(H,y)-L(H,x)= L^{i}(H,y)-L^{i}(H,x)-L^{j}(H,x)> L^{i}(H,z)-L^{j}(H,x)\geq {0},
$$

So that the maximum is attained for the subgraph $H'$ obtained from $H$
by deleting the vertex $j$ and the corresponding edges containing $j$. Similarly,
the theorem is again true
by the induction hypothesis.

{\bf{Case 2.}}  If $H$ is a complete $\{m,m-1\}$-graph, it is obvious that $\omega=n$ and the maximizer of $L(H, x)$ over $S$ must be a positive vector. Otherwise, we can also use the induction hypothesis. Then we have 
\begin{eqnarray*}
L(H)&=&\max\limits_{x\in S}L(H,x)=\max\limits_{x\in S}\sum\limits_{e\in E}\frac{1}{\alpha(s)}x_{m}^{e}=\max\limits_{x\in S}\sum\limits_{s\in{\{m-1,m\}}}\frac{1}{\alpha(s)}\sum\limits_{e\in{E}\atop {|e|=s}}x_{m} ^{e} \\
&\leq& \max\limits_{x\in S}\frac{1}{\alpha(m)}\sum\limits_{e\in{E}\atop {|e|=m}}x_{m} ^{e}+\max\limits_{x\in S}\frac{1}{\alpha(m-1)}\sum\limits_{e\in{E}\atop {|e|=m-1}}x_{m} ^{e}.\\
\end{eqnarray*}

It is easy to verify that $\alpha(m)=m!$ and $$\sum\limits_{e\in{E}\atop {|e|=m}}x_{m} ^{e}=\sum\limits_{\{i_{1},i_{2},\cdots,i_{m}\}\in E,\atop {1\leq i_{1}<i_{2}<\cdots<i_{k}\leq n}}m!x_{i_{1}}x_{i_{2}}\cdots x_{i_{m}},$$ where the sum is over all edges with cardinality $m$.

By Maclaurin's inequality, we have
\begin{eqnarray*}
\frac{1}{\alpha(m)}\sum\limits_{e\in{E}\atop {|e|=m}}x_{m} ^{e}&=&  \sum\limits_{\{i_{1},i_{2},\cdots,i_{m}\}\in E,\atop {1\leq i_{1}<i_{2}<\cdots<i_{k}\leq n}}x_{i_{1}}x_{i_{2}}\cdots x_{i_{m}}\\&\leq& \left(
  \begin{array}{ccc}
    n\\
    m
  \end{array}
\right)(\frac{\sum\limits_{i=1}^{n}x_{i}}{n})^{m}\\
&=&(\frac{1}{n})^{m}\left(
  \begin{array}{ccc}
    n\\
    m
  \end{array}
\right).
\end{eqnarray*}
The equality holds if and only if $x_{1}=x_{2}=\cdots=x_{n}=\frac{1}{n}.$

Similarly, $\alpha(m-1)=\sum\limits_{k_1,\cdots,k_{m-1}\geq 1,\atop k_1+\cdots+k_{m-1}=m} \frac{m!}{k_1!k_2!\cdots k_{m-1}!}=\frac{(m-1)m!}{2}$,
and \begin{eqnarray*}\sum\limits_{e\in{E}\atop {|e|=m-1}}x_{m} ^{e}&=&\sum\limits_{\{i_{1},i_{2},\cdots,i_{m-1}\}\in E}\sum\limits_{k_1,\cdots,k_{m-1}\geq 1,\atop k_1+\cdots+k_{m-1}=m} \frac{m!}{k_1!k_2!\cdots k_{m-1}!}x_{i_{1}}^{k_{1}}x_{i_{2}}^{k_{2}}\cdots x_{i_{m-1}}^{k_{m-1}}\\
&=&\frac{m!}{2}\sum\limits_{\{i_{1},i_{2},\cdots,i_{m-1}\}\in E}x_{i_{1}}x_{i_{2}}\cdots x_{i_{m-1}}(x_{i_{1}}+x_{i_{2}}+\cdots +x_{i_{m-1}})\\
&=&\frac{m!}{2}[x_{1}^{2}\sum\limits_{\{1,i_{2},\cdots,i_{m-1}\}\in E, \atop {2\leq i_{2}<\cdots<i_{m-1}\leq n}}x_{i_{2}}x_{i_{3}}\cdots x_{i_{m-1}}
+\cdots\\ &+&x_{n}^{2}\sum\limits_{\{i_{2},\cdots,i_{m-1},n\}\in E,\atop {1\leq i_{2}<\cdots<i_{k}\leq n-1}} x_{i_{2}}\cdots x_{i_{m-1}}]
\end{eqnarray*}

By Maclaurin's inequality, we have
\begin{eqnarray*}
\frac{1}{\alpha(m-1)}\sum\limits_{e\in{E}\atop {|e|=m-1}}x_{m} ^{e}&=& \frac{1}{m-1}[x_{1}^{2}\sum\limits_{\{1,i_{2},i_{3},\cdots,i_{m-1}\}\in E,\atop {2\leq i_{2}<\cdots<i_{m-1}\leq n}}x_{i_{2}}x_{i_{3}}\cdots x_{i_{m-1}}+\cdots\\
&+&x_{n}^{2}\sum\limits_{\{i_{2},\cdots,i_{m-1},n \}\in E,\atop {1\leq i_{2}<\cdots<i_{k}\leq n-1}} x_{i_{2}}\cdots x_{i_{m-1}}]\\
&\leq&\frac{1}{m-1}\left(
  \begin{array}{ccc}
    n-1\\
    m-2
  \end{array}
\right)[x_{1}^{2}(\frac{\sum\limits_{i=2}^{n}x_{i}}{n-1})^{m-2}+\cdots+x_{n}^{2}(\frac{\sum\limits_{i=1}^{n-1}x_{i}}{n-1})^{m-2}].\\
\end{eqnarray*}
The equality holds if and only if $x_{1}=x_{2}=\cdots=x_{n}=\frac{1}{n}.$

Hence, the maximum is attained on $x_{1}=x_{2}=\cdots=x_{n}=\frac{1}{n}$. After setting $x_{1}=x_{2}=\cdots=x_{n}=\frac{1}{n}$, we get
\begin{eqnarray*}
\max_{x\in S}\frac{1}{\alpha(m-1)}\sum\limits_{e\in{E}\atop {|e|=m-1}}x_{m} ^{e}&=& (\frac{1}{n})^{m-1}\frac{1}{m-1}\left(
  \begin{array}{ccc}
    n-1\\
    m-2
  \end{array}
\right)\\
&=&(\frac{1}{n})^{m}\frac{n}{m-1}\left(
  \begin{array}{ccc}
    n-1\\
    m-2
  \end{array}
\right)\\&=&(\frac{1}{n})^{m}\left(
  \begin{array}{ccc}
    n\\
    m-1
  \end{array}
\right)
\end{eqnarray*}

Combining the above inequalities, if $H$ is a complete $\{m,m-1\}$-graph, then $\omega=n$ and
\begin{eqnarray*}
L(H)&\leq&(\frac{1}{n})^{m}\left(
  \begin{array}{ccc}
    n\\
    m
  \end{array}
\right)+(\frac{1}{n})^{m}\left(
  \begin{array}{ccc}
    n\\
    m-1
  \end{array}
\right)=(\frac{1}{n})^{m}\left(
  \begin{array}{ccc}
    n+1\\
    m
  \end{array}
\right).
\end{eqnarray*}
The equality holds if and only if $x_{1}=x_{2}=\cdots=x_{n}=\frac{1}{n}$.

This completes the proof.
\end{proof}

\begin{theorem}\label{upper bound on clique and spectral radius}
Let $H$ be an $\{m,m-1\}$-graph with clique number $\omega$. If either $H$ is a complete $\{m,m-1\}$-graph or there exists two nonadjacent vertices $i$ and $j$ such that $R(i)=R(j)$, then
$$\rho(H)\leq m({\frac{U}{\omega}})^{m}\left(
  \begin{array}{ccc}
    \omega+1\\
    m
  \end{array}
\right)$$
where $U$ is the sum of the entries of the principal eigenvector.
\end{theorem}
\begin{proof} Let $x$ be the principal eigenvector of $H$. By Equation (\ref{spectral radius equation}), we have
\begin{eqnarray*}
\rho(H)&=&\sum\limits_{e\in E} x^\top \mathcal{A}(e)x=\sum\limits_{s\in \{m,m-1\}}\frac{s}{\alpha(s)}\sum\limits_{e\in E, \atop |e|=s}x^{e}_{m}.\end{eqnarray*}

Set $y=\frac{x}{U}$, where $U$ is the sum of the entries of $x$.
It is clearly that $y\in S$. Apply the Theorem \ref{general M-S}, then
\begin{eqnarray*}
\frac{\rho(H)}{U^{m}}&=&\sum\limits_{s\in \{m,m-1\}}\frac{s}{\alpha(s)}\sum\limits_{e\in E, \atop |e|=s}y^{e}_{m}\\
&\leq&\sum\limits_{s\in \{m,m-1\}}\frac{m}{\alpha(s)}\sum\limits_{e\in E, \atop |e|=s}y^{e}_{m}\\
&\leq& m({\frac{1}{\omega}})^{m}\left(
  \begin{array}{ccc}
    \omega+1\\
    m
  \end{array}
\right)
\end{eqnarray*}
Thus $$\rho(H)\leq m({\frac{U}{\omega}})^{m}\left(
  \begin{array}{ccc}
    \omega+1\\
    m
  \end{array}
\right)$$
Therefore the theorem follows.
\end{proof}
And further applying Lemma \ref{Cauchy¨CSchwarz inequality}, it is easy to check the following result.
\begin{corollary}
Let $H$ be an $\{m,m-1\}$-graph with clique number $\omega$. If either $H$ is a complete $\{m,m-1\}$-graph or there exists two nonadjacent vertices $i$ and $j$ such that $R(i)=R(j)$, then
$$\rho(H)\leq {\frac{mn^{m-1}}{\omega^{m}}}\left(
  \begin{array}{ccc}
    \omega+1\\
    m
  \end{array}
\right)$$
\end{corollary}

\begin{proof} Let $x$ be the principal eigenvector of $H$ and $U$ is the sum of the entries of $x$. By Lemma \ref{Cauchy¨CSchwarz inequality}, we have
\begin{eqnarray*}
U&=&\sum\limits_{i=1}^{n}x_{i}=\sum\limits_{i=1}^{n}1\cdot 1\cdots x_{i}\leq (\sum\limits_{i=1}^{n}1)^{\frac{m-1}{m}}(\sum\limits_{i=1}^{n}{x_{i}^{m}})^{\frac{1}{m}}=n^{\frac{m-1}{m}}.\end{eqnarray*}
Obviously, $U^{m}\leq n^{m-1}.$ Combining with Theorem \ref{upper bound on clique and spectral radius}, the result follows.

\end{proof}

\vskip 3mm
\end{spacing}

\begin{thebibliography}{99}

\bibitem{Berge1973} C. Berge, Hypergraph: Combinatorics of Finite Sets, third edition, North-Holland, Amsterdam, 1973.

\bibitem{Bretto2013_2} A. Bretto, Hypergraph Theory: An Introduction, Springer, 2013.

\bibitem{Banerjee2017} A. Banerjee, A. Char, B. Mondal, Spectral of general hypergraphs, Linear Algebra Appl. 518 (2017) 14-30.

\bibitem{Cooper2012} J. Cooper, A. Dutle, Spectra of uniform hypergraphs, Linear Algebra Appl. 436 (2012) 3268-3292.

\bibitem{Fan2015}Y. Fan, Y. Tan, X. Peng, A. Liu, Maximizing spectral radii of uniform hypergraphs with few
edges, Discussiones Math. Graph Theory  36(2016) 845-856.

\bibitem{Frankl1998} P. Frankl, Z. F\"{u}redi, Extremal problems and the Lagrange function of hypergraphs, Bulletin Institute Math.
Academia Sinica 16(1988) 305-313.

\bibitem{Yuejian Peng2016} R. Gu, X. Li, Y. Peng, Y. Shi, Some Motzkin-Straus type results for non-uniform
hypergraphs, J. Comb. Optim. 31(2016) 223-238.

\bibitem{Hardy.Littlewood1988_4}G. Hardy, J. Littlewood, G. P\'{o}lya, Inequalities, 2nd edition, Cambridge University Press, 1988.

\bibitem{Keevash2013} D. Hefetz, P. Keevash, A hypergraph Tur\'{a}n theorem via lagrangians of intersecting families, J. Combin.
Theory Ser. A  120 (2013) 2020-2038.

\bibitem{Macdonald1963} J. MacDonald Jr., Problem E1643, Amer. Math. Monthly 70 (1963) 1099.

\bibitem{Motzkin1965} T. Motzkin, E. Straus, Maxima for graphs and a new proof of a theorem of Tur\'{a}n, Canad. J. Math. 17(1965) 533-540.

\bibitem{Mubayi2006} D. Mubayi, A hypergraph extension of Turans theorem, J. Combin. Theory Ser. B 96 (2006) 122-134.

\bibitem{Nikiforov2014} V. Nikiforov, Analytic methods for uniform hypergraphs, Linear Algebra Appl. 457 (2014) 455-535.

\bibitem{Papendieck2000} B. Papendieck, P. Recht, On maximal entries in the principal eigenvector of graphs, Linear Algebra Appl. 310 (2000) 129-138.

\bibitem{peng2016}  Y. Peng, H. Peng, Q. Tang, C. Zhao, An extension of Motzkin-Straus Thorem to non-uniform hypergraphs
and its applications, Discrete Appl. Math. 200 (2016) 170-175.

\bibitem{Qi2013} L. Qi, Symmetric nonnegative tensors and copositive tensors, Linear Algebra Appl. 439(2013) 228-238.

\bibitem{Lim2005_3} L. Lim, Singular values and eigenvalues of tensors: a variational approach, in: Proceedings
of the IEEE International Workshop on Computational Advances in Multi-Sensor Adaptive
Processing (CAMSAP 05) 1(2005) 129-132.

\bibitem{Liu2016} L. Liu, L. Kang, X. Yuan, On the principal eigenvectors of uniform hypergraphs, Linear Algebra Appl. 511(2016) 430-446.

\bibitem{Qi2005_2} L. Qi, Eigenvalues of a real supersymmetric tensor, J. Symb. Comput. 40(2005) 1302-1324.

\bibitem{Rota 2009} S. Rota Bul\`{o}, M. Pelillo, A generalization of the Motzkin-Straus theorem to hypergraphs, Optim. Lett. 3 (2009) 187-295.

\bibitem{Pelillo 2009} S. Rota Bul\`{o}, M. Pelillo, New bounds on the clique number of graphs based on spectral
hypergraph theory, Learning and Intelligent Optim.
 5851(2009) 45-58.

\bibitem{Talbot2002} J. Talbot, Lagrangians of hypergraphs, Combin. Prob. Comput. 11 (2002) 199-216.

\bibitem{Wilf1967} H. Wilf.,  The eigenvalues of a graph and its chromatic number, J. London
Math. Soc. 42(1967) 330-332.

\bibitem{Wilf1986} H. Wilf., Spectral bounds for the clique and independence numbers of graphs,
J. Comb. Theory Series B 40(1986) 113-117.

\bibitem{chang2013} G. Yi, A. Chang, The spectral bounds for the clique numbers of $r$-uniform hypergraphs,
Manuscript, Fuzhou University, 2013.

\bibitem{Kang2017} W. Zhang, L. Liu, L. Kang, Y. Bai, Some properties of the Spectral radius for general hypergraphs, Linear Algebra Appl.
 513(2017) 103-119.



\end{thebibliography}
\end{document}